\documentclass[10pt]{amsart}
\usepackage[margin=1.2in,marginparsep=0.1in,marginparwidth=1in]{geometry}
\usepackage{amssymb,amsmath,amsthm,amstext,amscd,latexsym,graphics,graphicx,bbm,caption}
\usepackage[usenames,dvipsnames,svgnames,table]{xcolor}
\usepackage[plainpages=false,colorlinks=true, pagebackref]{hyperref}
\usepackage{tikz}
\usepackage{tikz-cd}
\usetikzlibrary{positioning}

\tikzset{>=stealth}
\hypersetup{citecolor=Sepia,linkcolor=blue, urlcolor=blue}

\usepackage{cite}   

\makeatletter
\def\@tocline#1#2#3#4#5#6#7{\relax
  \ifnum #1>\c@tocdepth 
  \else
    \par \addpenalty\@secpenalty\addvspace{#2}%
    \begingroup \hyphenpenalty\@M
    \@ifempty{#4}{%
      \@tempdima\csname r@tocindent\number#1\endcsname\relax
    }{%
      \@tempdima#4\relax
    }%
    \parindent\z@ \leftskip#3\relax \advance\leftskip\@tempdima\relax
    \rightskip\@pnumwidth plus4em \parfillskip-\@pnumwidth
    #5\leavevmode\hskip-\@tempdima
      \ifcase #1
       \or\or \hskip 2em \or \hskip 2em \else \hskip 3em \fi%
      #6\nobreak\relax
    \dotfill\hbox to\@pnumwidth{\@tocpagenum{#7}}\par
    \nobreak
    \endgroup
  \fi}
\makeatother

\newtheorem{intro-thm}{Theorem}[]
\theoremstyle{plain}
\newtheorem{thm}{Theorem}[section]
\newtheorem{theorem}[thm]{Theorem}

\newtheorem{lemma}[thm]{Lemma}

\newtheorem{proposition}[thm]{Proposition}

\theoremstyle{definition}
\newtheorem{remark}[thm]{Remark}

\newcommand{\Spec}{{\rm Spec \,}}

\renewcommand{\tilde}{\widetilde}

\newcommand{\sD}{{\mathcal D}}
\newcommand{\sE}{{\mathcal E}}

\newcommand{\sH}{{\mathcal H}}

\newcommand{\sO}{{\mathcal O}}

\newcommand{\sS}{{\mathcal S}}

\newcommand{\A}{{\mathbb A}}

\renewcommand{\P}{{\mathbb P}}

\newcommand{\etale}{\'{e}tale}

\newcommand{\ra}{\rightarrow}

\input{xy}
\xyoption{all}

\begin{document}

\title{G\MakeLowercase{abber's presentation lemma over noetherian domains}}
\author{{N\MakeLowercase{eeraj} D\MakeLowercase{eshmukh,}}\ {A\MakeLowercase{mit} H\MakeLowercase{ogadi,}} 
{G\MakeLowercase{irish} K\MakeLowercase{ulkarni}}\ \MakeLowercase{and}\ {S\MakeLowercase{uraj} Y\MakeLowercase{adav}}}



\date{}

\begin{abstract} Following Schmidt and Strunk, we give a proof of Gabber's presentation lemma over a noetherian domain with infinite residue fields.
\end{abstract}

\maketitle

\section{Introduction}
Gabber's presentation lemma, initially proved by O. Gabber for the base, spectrum of an infinite field in \cite{gabber} (see also \cite{colliot}, \cite{hogadi2018}) plays a fundamental role in the study of $\A^1$- homotopy theory, especially as developed by Morel in \cite{morel2012}. 
This lemma may be thought of as an algebro-geometric analogue of the tubular neighbourhood theorem in differential geometry. In \cite{schmidt2018}, this lemma was generalized by J. Schmidt and F. Strunk to the case where the base is a spectrum of a Dedekind domain with infinite residue fields. The goal of this paper is to show that the arguments given in \cite{schmidt2018} can, in fact, be modified to obtain a proof of Gabber's presentation lemma over a general noetherian domain with all its residue fields infinite. The following is the main result of this paper.

\begin{thm}\label{main}
	Let $S=\Spec (R)$ be the spectrum of a noetherian domain with all its residue fields infinite. Let $X$ be a smooth, irreducible, equi-dimensional $S$-scheme of relative dimension $d$. Let $Z\subset X$ be a closed subscheme, $z$ be a closed point in $Z$ lying over $s\in S$, such that that $dim (Z_s) < dim (X_s)$. Then after possibly replacing $S$ by a Nisnevich neighbourhood of $s$ and $X$ by a Nisnevich neighbourhood of $z$, there exists a map $\Phi=(\Psi,\nu):X\rightarrow \A_{S}^{d-1}\times \A_{S}^1$, an open subset $V\subset\A_{S}^{d-1}$ and an open subset $U\subset \Psi^{-1}(V)$ containing $z$ such that
	\begin{enumerate}
		\item $Z\cap U = Z\cap \Psi^{-1}(V)$		
		\item $\Psi_{|Z}: Z\rightarrow \A_{S}^{d-1}$ is finite
		\item $\Phi_{|U}: U\rightarrow \A_{S}^d$ is \etale{} 
		\item $\Phi_{|Z\cap U}:Z\cap U\rightarrow \A^1_V$ is a closed immersion 
		\item $\Phi^{-1}(\Phi(Z\cap U))\cap U=Z\cap U$.
	\end{enumerate}
\end{thm}

     In \cite{schmidt2018} J. Schmidt and F. Strunk, use the presentation lemma to generalize the $\A^1$-connectivity result of F. Morel ( \cite[Theorem 6.1.8]{morel2005}) over Dedekind schemes with infinite residue fields. As an application of Theorem \ref{main}, we observe that the connectivity result holds over any noetherian domain with all its residue fields infinite. To state this result we recall the following standard notation: For a base scheme $S$, let $\sS\sH^s_{S^1}(S)$
     be the model category of sheaves of $S^1$-spectra over $S$. For an integer $i$, let $\sS\sH^{s}_{S^1 \geq i}(S)$ be 
     the full subcategory of $i$-connected spectra. Let 
     $ \sS\sH^s_{S^1}(S) \xrightarrow{L^{\A^1}} \sS\sH^s_{S^1}(S)$
     be the $\A^1$-fibrant replacement functor. Then
 

\begin{theorem}\label{connectivity}
	Let $S=\Spec (R)$ be the spectrum of a noetherian domain of dimension $d$ with all its residue fields infinite. Then S has the shifted stable $\A^1$-connectivity property, that is, if $E \in \sS\sH^{s}_{S^1 \geq i}(S)$ then $L^{\A^1} E \in \sS\sH^{s}_{S^1 \geq i-d}(S).$
\end{theorem}
The proof of Theorem \ref{connectivity} is exactly the same as the proof of its analogue in \cite{schmidt2018}  except for the input from Gabber's presentation lemma, the required generality of which is available once Theorem \ref{main} is proved. We present a sketch of the proof of Theorem \ref{connectivity} in Section 4.\\

An important ingredient of the proof of the Gabber's presentation lemma of \cite{schmidt2018} is \cite[Theorem 4.1]{kai2015}, which states that given an equi-dimensional scheme $Y$ over a Dedekind scheme $B$ with infinite residue fields, Nisnevich locally on $B$ there exists a projective closure $\overline{Y}$ of $Y$ in which $Y$ is fiber-wise dense. Unfortunately, we are unable to prove such a result over a general base. However, we observe that a slightly weaker result (see Theorem \ref{ydense}) can be proved which suffices for our purpose. As in Gabber's original proof of the presentation lemma, as well as in \cite{schmidt2018}, the condition of residue fields being infinite in Theorem \ref{main} is required in order to make suitable generic choices. We are currently working on removing the condition of residue fields being infinite taking inputs from \cite{hogadi2018}.

\noindent{\bf Acknowledgments.} The first-named author was supported by the INSPIRE fellowship of the Department of Science and Technology, Govt.\ of India during the course of this work. The last-named author was supported by NBHM fellowship of Department of Atomic Energy, Govt.\ of India during the course of this work.  We thank F. Strunk and  J. Koll\'ar for their comments on the draft of this paper. We thank the anonymous referees for their helpful comments and suggestions.

\section{Fiber-wise denseness}\label{section_dense}
In this section, we prove a technical result which is crucial to the proof of our main theorem. It is essentially \cite[Theorem 4.1]{kai2015} with minor modifications (see also \cite[Theorem 10.2.2]{levine2006} ). Throughout this section, $\dim_B(Y)$ denotes the supremum of dimensions of all the fibers of $Y\rightarrow B$.

\begin{theorem}\label{ydense}
Let $B$ be the spectrum of a noetherian domain.	Let $Y/B$ be either a smooth scheme or a divisor in a smooth scheme $X$. Let $y\in Y$ be a point lying over a point $b\in B$ with $\dim_B(Y_b) = n$. Assume $k(b)$ is an infinite field. Then there exist Nisnevich neighborhoods $(Y',y)\ra(Y,y)$ and $(B',b)\ra(B,b)$, fitting into the following commutative diagram 
		\begin{center}
		\begin{tikzcd}
		& Y' \arrow[d] \arrow[r]&Y\arrow{d}\\
		& B'\arrow[r]&{B}
		\end{tikzcd}
	\end{center}
	and a closed immersion $Y'\ra \A_{B'}^N$ for some $N \geq 0$ such that if $\overline{Y'}$ is its closure in $\P_{B'}^N$ then $Y'_y$ is dense in the union of $n$-dimensional irreducible components of $(\overline{Y'})_y$. 
\end{theorem}

\begin{remark}\label{Kaiapp}
	The above theorem is a weaker statement than \cite[Theorem 4.1]{kai2015} (see also \cite[Theorem 10.2.2]{levine2006}) but over a general base. In the proof of \cite[Theorem 4.1]{kai2015} the author mentions that the base is assumed to be Dedekind to ensure that the projective closure of an equi-dimensional scheme remains equi-dimensional over $B$.

\end{remark}
We begin with an intermediary lemma which will be used repeatedly (see also \cite[Lemma 10.1.4]{levine2006}).

\begin{lemma}\label{1-d}
Let $X$ be an affine scheme. Choose a closed embedding $X\rightarrow\A^N_B$ and a point $x\in X$. Let $\overline{X}$ be the projective closure of $X$ in $\P^N_B$, and assume that it has fiber dimension $n$. Then, there exists 
\begin{enumerate}
\item a projective scheme $\tilde{X}$, 
\item an open neighbourhood $X_0$ of $x$ (in $X$), 
\item an open immersion $X_0\hookrightarrow \tilde{X}$ and 
\item a projective morphism $\psi: \tilde{X}\rightarrow \P^{n-1}_B$
\end{enumerate}
such that $\psi$ has fiber dimension one.
\end{lemma}
\begin{proof}
We follow the arguments given in \cite[Theorem 4.1]{kai2015} verbatim (see also \cite[Theorem 10.1.4]{levine2006}).
After possibly shrinking $B$, we can find $n$ hyperplanes $\Psi=\lbrace\psi_1,\ldots,\psi_n\rbrace$ which are part of a basis of $\Gamma(\P^N_B,\sO(1))$ as a $B$-module. The choice is such that $V(\Psi)$, which denotes the common zeros of all $\psi_i$, does not contain $x$ and it meets $X$ fiber-wise properly over $B$, so that $\overline{X}\cap V(\Psi)$ is finite over $B$. Let $p:\tilde{\P^N}\rightarrow \P^N$ be the blowup of $\P^N$ along $V(\Psi)$, and $\tilde{X}$ the strict transform of $\overline{X}$ in the blowup. Let $\psi: \tilde{\P^N_B}\rightarrow \P^{n-1}_B$ denote the map induced by the rational map defined by a projection from $V(\Psi)$. Let $X_0:=\overline{X}\setminus V(\Psi)$. We have the following commutative diagram:
\begin{center}
		\begin{tikzcd}[row sep=tiny]
		& \tilde{X} \arrow[dd] \arrow[r,hook,"cl."]&\tilde{\P^N_B}\arrow{dd}\arrow[r,"\psi"]&\P^{n-1}_B\\
		X_0 \arrow[ur,hook] \arrow[dr,hook] & \\
		& \overline{X}\arrow[r,hook,"cl."]&{\P^N_B}
		\end{tikzcd}
\end{center}
We claim that $\psi:\tilde{X}\rightarrow\P^{n-1}_B$ has fiber dimension one.
To see this, choose any point $y\in\P^{n-1}_B$, and consider the composite $a:\Spec(\Omega)\overset{y}{\rightarrow}\P^{n-1}_B\rightarrow B$. Then, the fiber of $\psi$ over $y$ may be identified with a linear subscheme $V(y)$ of $\P^N_a$, of dimension $N-n+1$. Furthermore, $V(y)$ contains the base change $V(\Psi)_a$, which has dimension $N-n$, by construction. Again by construction, the intersection $V(y)\cap \overline{X}\cap V(\Psi)_a$ is finite in $\P^N_a$. This means that $V(y)\cap \overline{X}$ has dimension $1$ in the projective space $V(y)$.

Further note that for $x\in V(\Psi)$, $p^{-1}(x)\simeq\P^{n-1}$. Also, the exceptional divisor of $\tilde{X}$ is an irreducible subscheme. Therefore, for any point $x\in V(\Psi)\cap X$, the fiber $\tilde{X}_x$ is an irreducible subscheme of $\P^{n-1}$ of dimension $n-1$. Therefore, $p^{-1}(\overline{X})=\tilde{X}$, so that $p:\psi^{-1}(y)\cap \tilde{X}\rightarrow V(y)\cap \overline{X}$ is a bijection. Thus, $\psi: \tilde{X}\rightarrow \P^{n-1}_B$ has 1-dimensional fibers.
\end{proof}

	

\begin{proof}[Proof of \ref{ydense}] We first prove the result in the case when $Y=X$ is a smooth scheme. The proof is by induction on $n$. The case $n=0$ follows from a version of Hensel's lemma.\\
	\noindent \underline{Step 1}: As $X$ is smooth, Zariski locally on $B$, we write $X$ as a hypersurface in some $\A^N_B$. Let $\overline{X}$ denote its reduced closure in $\P^N_B$. Note that $\overline{X}$ also has fiber-dimension $n$ over $B$. By applying Lemma \ref{1-d}, we get a projective morphism $\psi: \tilde{X}\rightarrow \P^{n-1}_B$ with 1-dimensional fibers.


\noindent\underline{Step 2}: Set $T= \P^{n-1}_B$ and $t=\psi(x)$. Choose any projective embedding $\tilde{X} \hookrightarrow \P^{N_2}_T$. Let $(\tilde{X})_t$ and $(X_0)_t$ denote the fibers over $t$ of $\tilde{X}$ and $X_0$ respectively. Then choose a hypersurface $H_t \subset \P^{N_2}_t$ satisfying the next three conditions. 
\begin{enumerate}
	\item $x\in H_t$ (if $x$ is a closed point in $(X_0)_t$) 
	\item $(\tilde{X})_t$ and $H_t$ meet properly in $\P_t^{N_2}$.
	\item $H_t$ does not meet $\overline{({X_0})_t}\setminus{({X_0})_t}$.
\end{enumerate}

Now after restricting to a suitable Nisnevich neighbourhood of $T$, which we denote again by $T$ (and after base changing everything to $T$), using the hyperplane $H_t$, we can choose a Cartier divisor $\sD$ which fits into the following diagram 
\begin{center}
	\begin{tikzcd}[row sep=tiny,column sep=large]
	& \tilde{X}  \arrow{r}{projective}[swap]{1-dim}&T\arrow[r,"Nis"]&\P^{n-1}_B\\
	X_0 \arrow[ur,hook] & \\
	& \sD\arrow[uu]\arrow[ul,hook,"{Cartier. div}"]\arrow{uur} [swap]{finite}
	\end{tikzcd}
\end{center}
For sufficiently large $m$ we can find a section $s_0$ of $\Gamma(\tilde{X},\sO_{\tilde{X}}(m\sD))$ which maps to a nowhere vanishing section of $\Gamma(\sD,\sO_{\sD})$. Let $s_1: \sO_{\tilde{X}} \ra \sO_{\tilde{X}}(m\sD)$ be the canonical inclusion. Since the zero-loci of $s_0$ and $s_1$ are disjoint, we get a map
$$f=(s_0,s_1):\tilde{X} \ra \P^1_{T}.$$

Since the quasi-finite locus of a morphism is open, shrink $T$ around $t$ such that $\sD$ is contained in the quasi finite locus of $f$ after the base change. Let $X_0'$ be the quasi-finite locus of the base change.

	\begin{center}
	\begin{tikzcd}[column sep=large]
	& f^{-1}(\infty_T)=\sD \arrow[d] \arrow[r,hook]&X_0'\arrow{d}[swap]{{quasi-finite}}\arrow[r,hook]&\tilde{X}\arrow[dl,"f"]\\
	& \infty_T\arrow[r,hook]&{\P^1_T}
	\end{tikzcd}
\end{center}
 Then the subset $W=f(\tilde{X}\setminus X_0')\subset \P^1_T$ is proper over $T$ and is contained in $\P^1_T\setminus f(H_t)=\A^1_T$. Hence, it is finite over $T$. The map $\tilde{X}\setminus f^{-1}(W)\ra \P^1_T\setminus W$, being proper and quasi-finite, is finite. By condition $(1)$, we see that $\tilde{X}\setminus f^{-1}(W)$ contains $x$.\\
	\noindent\underline{Step 3}: Now by induction there exist Nisnevich neighborhoods $B_1\ra B$ and $T_1 \ra T$ such that the projective compactification $T_1\ra \overline{T_1}$ is fiber-wise dense in the union of $n$-dimensional irreducible components over $B_1$. Take a factorization of $f$ of the form $\tilde{X} \hookrightarrow \P^{N_3}_{T_1}\times_{T_1}\P^1_{T_1} \ra \P^1_{T_1}$. Let $\overline{X_1}$ denote the reduced closure of $\tilde{X}$ in $\P^{N_3}_{\overline{T_1}}\times_{\overline{T_1}}\P^1_{\overline{T_1}}$. We get the following diagram where every square is Cartesian
	\begin{center}
	\begin{tikzcd}[row sep=large]
	X_2:=\tilde{X}\setminus f^{-1}(W) \arrow[d,hook] \arrow[r,hook]&\tilde{X}\arrow[d,hook]\arrow[r,hook]&\overline{X_1}\arrow[d,hook]\\
	\P^{N_3}_{T_1}\times_{T_1}(\P^1_{T_1}\setminus W) \arrow[r,hook]\arrow{d}&\P^{N_3}_{T_1}\times_{T_1}\P^1_{T_1} \arrow[r,hook]\arrow{d}&\P^{N_3}_{\overline{T_1}}\times_{\overline{T_1}}\P^1_{\overline{T_1}}\arrow{d}\\
	\P^1_{{T_1}}\setminus W \arrow[r,hook] &\P^1_{{T_1}}\arrow[r,hook]&\P^1_{\overline{T_1}}
	\end{tikzcd}
		\end{center}
	
	By Stein factorization we decompose the map $\overline{f_1}:\overline{X_1}\ra \P^1_{\overline{T_1}}$ as 
	$$ \overline{f_1}:\overline{X_1}\ra \overline{X_2}\xrightarrow{finite} \P^1_{\overline{T_1}},$$ where the first map has geometrically connected fibers. Since $\overline{f_1}$ is finite over the open set $\P^1_{T_1}\setminus W$, 	$\overline{X_2}\times_{\P^1_{\overline{T_1}}}(\P^1_{T_1}\setminus W)$ is isomorphic to $X_2 :=\tilde{X}\setminus f^{-1}(W) $. Since ${X_2}$ is open in $\overline{X_2}$, the fiber dimension of $\overline{X_2}$ is at least $n$. Combining this with the fact that $\overline{X_2}$ is finite over $\P^1_{\overline{T_1}}$, we conclude that the fiber dimension of $\overline{X_2}$ over $B_{1}$ is exactly $n$.
	
	We observe that since $T_{1}$ is fiberwise dense in the union of $n$-dimensional irreducible components of $\overline{T_{1}}$, so is $\P^1_{{T_1}}$ (in $\P^1_{\overline{T_1}}$). Also as $W$ is finite over $T_{1}$, $\P^1_{T_1}\setminus W$ is fiberwise dense in $\P^1_{T_1}$. Hence it is dense in the union of $n$-dimensional irreducible components of $\P^1_{\overline{T_1}}$.
	Now we claim that $X_2 $ intersects the fiber of $\overline{X_2}$ over any point $b_{1}$ of $B_{1}$. Let $X_{2}'$ be an $n$-dimensional irreducible component of the fiber $(\overline{X_2})_{{b}_{1}}$. Then the induced map $X_{2}' \rightarrow( \P^1_{\overline{T_1}})_{{b}_{1}}$ is a finite morphism of schemes of the same dimension. Hence it is a surjection to an irreducible component say, $U$ of $( \P^1_{\overline{T_1}})_{{b}_{1}}$. Further $\P^1_{T_1}\setminus W$ intersects $U$ by denseness. Taking inverse image of its intersection with irreducible component proves that $X_{2}$ intersects $Y$.
	
	As $\overline{X_2}$ is projective over $B_{1}$, we choose any embedding of it in projective space $\P^{N}_{{B}_{1}}$. Then for the closed subscheme $\overline{X_2} \setminus X_2$ (with reduced structure) there exists a hypersurface $H$ of $\P^{N}_{{B}_{1}}$ of degree, say $d$, containing $\overline{X_2} \setminus X_2$, not containing the point $x$ and such that $H_{{b}_{1}}$ intersects $({X_2})_{{b}_{1}}$ properly in $\P^{N}_{{b}_{1}}$. Hence by discussion in previous paragraph, $H_{{b}_{1}}$ also intersects $(\overline{X_2})_{{b}_{1}}$ properly. Replacing $X_{2}$ by $\overline{X_2} \setminus H$ and taking $d$ fold Veronese embedding we may assume $H$ to be $\P^{N-1}_{\infty}$. Now we have the embedding $\overline{X_2} \setminus H \hookrightarrow \A^{N}_{{B}_{1}} = \P^{N}_{{B}_{1}} \setminus \P^{N-1}_{\infty} $ thereby proving the smooth case.

We shall now consider the case when $Y$ is a divisor in a smooth scheme.

\noindent\underline{Step 4}: Let $Y$ be a divisor in a smooth scheme $X$. We will produce a map, $\psi:\tilde{Y}\rightarrow \P^{d-1}$ whose fibers are $1$-dimensional.

Since $X$ is smooth, by Steps 1-3, Nisnevich locally, we have a closed embedding $Y\rightarrow X\rightarrow\A^N_B$ such that all fibers of $\overline{Y}\rightarrow B$ are $n$-dimensional. Then by Lemma \ref{1-d}, we have a commutative diagram,
\begin{center}
		\begin{tikzcd}[row sep=tiny]
		& \tilde{Y} \arrow[dd] \arrow[r,hook,"cl."]&\tilde{\P^N_B}\arrow{dd}\arrow[r,"\psi"]&\P^{n-1}_B\\
		Y_0 \arrow[ur,hook] \arrow[dr,hook] & \\
		& \overline{Y}\arrow[r,hook,"cl."]&{\P^N_B}
		\end{tikzcd}
\end{center}
Then as in Step 2 above, we obtain a morphism Nisnevich locally on $Y$, $\phi : Y \rightarrow  \P^{1}_{T}$, where $T$ is a Nisnevich neighbourhood of $\P^{n-1}$. 
	Since $T$ is a smooth $B$-scheme, our theorem holds for $T$. 
	Remaining proof is the same as in Step 3.
\end{proof}

\section{Relative version of Gabber's Presentation Lemma }
We now prove Theorem \ref{main}. We follow \cite{schmidt2018} to prove Theorem \ref{main}, the only difference being, that in their version of Theorem \ref{ydense} (which is for Henselian DVR), they have the stronger condition of fiberwise denseness, which they use to construct a finite map $\Psi_{|Z}: Z\rightarrow \A_{S}^{d-1}$. However, we observe that their proof still goes through with our weaker condition of denseness in $n$-dimensional components, which we illustrate in Propositions \ref{misses} and \ref{finite}. The rest of the proof does not require any new inputs and we just state those results from \cite{schmidt2018} which are essentially an application of the proof from \cite{colliot}.\\


First we reduce to the case that $z$ is a closed point and $Z$ is a principal divisor. 
\begin{lemma}\label{reduction}(See \cite[Lemma 3.2.1]{colliot})
With the notation as in Theorem \ref{main}, there exists a closed point $z'\in X$ such that $z'$ is a specialization of $z$ and there exists a non-zero $f\in \Gamma(X,\sO_X)$ such that $Z\subset V(f)$.
\end{lemma}

\begin{remark}
Since in Theorem \ref{main}, we assume that $dim (Z_s) < dim (X_s)$, in the Lemma \ref{reduction} we furthermore assume $f$ is such that $dim((V(f)_{s})< dim (X_{s})$.
\end{remark}

\begin{remark}
Since Theorem \ref{main} is a Nisnevich local statement, henceforth we assume that the ring $R$ is Henselian local with closed point $\sigma$ and infinite residue field $k$.
\end{remark}


Let $S= Spec(R)$ with $\A^{n}_S = R[x_1,\ldots,x_n]$. Let $E$ be the $R$-span of $\{x_1,\ldots,x_n\}$ and consider $\sE:=\underline{\Spec}(Sym^{\bullet}E^{\vee})$ (note that $\sE(R)=E$). For any integer $d > 0$ and any $R$-algebra $A$, $\sE^d(A)$ parametrizes all linear morphisms $v= (v_1,\ldots,v_{d}):\A^{n}_T\rightarrow \A^{d}_T$, where $T= Spec(A)$. Considering $\A^n_S \hookrightarrow \P^{n}_S = $ Proj $S[X_0,\ldots,X_n]$, as a distinguished open subscheme $D(X_{0})$, we extend such a linear morphism to a rational map $ \overline{v}:\P^{n}_S \dashrightarrow  \P^{d}_S$ whose locus of indeterminacy $L_{v}$ is given by the vanishing locus of $v_1,\ldots,v_{d}$ and $X_0$, $V_{+}(X_0,v_1,\ldots,v_{d})  \subseteq \P^{n}_S$ (We will use this notation throughout this section). Given any closed subscheme $Y$ in $\A^n_S$, we denote by $\overline{Y}$ its projective closure in $ \P^{n}_S $. For the following lemma we refer to \cite[Lemma 2.3]{schmidt2018}
\begin{lemma}(see \cite[Lemma 2.3]{schmidt2018})\label{shafar}
	In the setting of the previous paragraph if $L_{v} \cap \overline{Y} = \emptyset$, then $\overline{v}:\overline{Y} \rightarrow  \P^{d}_S$ and $v: Y \rightarrow \A^d_S$ are finite maps.
\end{lemma}



Following lemma is standard.

\begin{lemma}\label{dimdrop}
	Let $W$ be a closed subscheme of $\P^N_k$. Then there exists a hyperplane $H \subset \P^N_k$ such that $\dim_k(H\cap W)=\dim_k(W)-1$.
\end{lemma}
\begin{proof}
	Let $\zeta_1,\ldots\zeta_r$ be the generic points of $W$ corresponding to the homogeneous prime ideals $\wp_{1},\ldots \wp_{r} $. Viewing the $\wp_{i}$'s and $\varGamma(\sO(1),\P^N_k )$ as vector spaces over the infinite field $k$, we can find a hyperplane $H$ not containing $\zeta_i$'s: as no non trivial vector space over an infinite field can be written as a finite union of proper subspaces. Hence by Krull's principal ideal theorem $\dim_k(H\cap W)=\dim_k(W)-1$. 
\end{proof}


\begin{proposition}\label{misses}
	Let $Y $ be as in Theorem \ref{ydense} and $\overline{Y}$ be its projective closure, then there exist $v_1,\ldots,v_{n}$ in the $k$-span of $\{X_1,\ldots,X_N\}$ such that $(\overline{Y})_{\sigma}\cap L_v= \emptyset$, where $L_v= V_{+}(X_0,v_1,\ldots,v_{n})$.
\end{proposition}
\begin{proof}
	Without loss of generality, we assume $\A^{n}_k = D(X_{0})$. Let $H_{\infty} = V_{+}(X_{0})$ denote the hyperplane at infinity of $\P^{N}_{k}$. Generic points of irreducible components of $\overline{Y_{\sigma}}$ lie in $\A^{n}_k = D(X_{0})$. Therefore $\dim(\overline{Y_{\sigma}} \cap H_{\infty} ) = n-1$. By Theorem \ref{ydense}, we have $\dim((\overline{Y})_{\sigma} \cap H_{\infty} ) = n-1$. Now applying Lemma \ref{dimdrop} repeatedly proves the claim.
\end{proof}



\begin{theorem}\label{fer}
	Let $X=Spec(A)/S$ be a smooth, equi-dimensional, affine, irreducible scheme of relative dimension $d$. Let $Z=Spec(A/f)$, $z$ be a closed point in $Z$ lying over $s\in S$, where $f$ is such that $\dim (Z_s) < \dim (X_s)$. Then there exists an open subset $\Omega\subset \sE^d$ with $\Omega(k)\neq \emptyset$ such that for all $\Phi=(\Psi,\nu)\in \Omega(k)$ the following hold
	\begin{enumerate}
	\item $\Psi_{|Z}:Z\ra \A^{d-1}_S$ is finite.
	\item $\Psi$ is \etale{} at all points of $F:=\psi^{-1}(\psi(z))\cap Z$.
	\item $\Phi_{|F}:F\ra \Phi(F)$ is radicial.
	\end{enumerate} 
\end{theorem}
Recall that $\Phi:F\ra \Phi(F)$ is said to be radicial \cite[Tag 01S2]{stacksradicial} if $\Phi$ is injective and for all $x \in F$ the residue field extension $k(x)/k(\Phi(x))$ is trivial.\\
To prove this theorem, we first get an open set of finite maps in Proposition \ref{finite}. Then we get a non-empty open set of \etale{} and radicial maps in Lemma \ref{er}. 


\begin{remark}\label{rem}
	By \cite[Prop. 2.6 and Lemma 2.7]{schmidt2018} we have a closed embedding $X \hookrightarrow \A^{N}_S$ such that $Z$ (Nisnevich locally around $z$) satisfies Theorem \ref{ydense}. 
\end{remark}

\begin{proposition}\label{finite}
	Let $X$ and $Z$ be as in Theorem $\ref{fer}$ with $S$ a spectrum of a Henselian local ring $R$. Then there is an open subset $\Omega\subset \sE^{d}$ with $\Omega(R)\neq \emptyset$ such that for all $\Psi \in \Omega(R)$, $\Psi_{|Z}:Z\ra \A^{d-1}_S$ is finite.
\end{proposition}
\begin{proof}
  We proceed as in \cite[Lemma 2.11]{schmidt2018}. By Remark \ref{rem} we have closed embedding $X \hookrightarrow \A^{N}_S$. Viewing $\sE^{d-1}$ as a closed subscheme of $\sE^{d}$ by taking the first $d-1$ factors we	consider the closed subscheme 
\begin{center}
	$V = \sE^{d-1}  \times_{S} H_{\infty}  \hookrightarrow \sE^{d} \times_{S} H_{\infty}$
\end{center}
where $H_{\infty}$ is the hyperplane at infinity in $\P^{N}_{S}$. Note that $V \rightarrow \sE^{d}$ has fiber $V_{v}= L_{(v_1,\ldots,v_{d-1})}$ for any $v= (v_1,\ldots,v_{d}) \in \sE^{d}(R)$. Consider the open subscheme $\Omega$ of $\sE^{d}$ defined as 

	$$\sE^{d} \setminus p_{1}(V \cap (\sE^{d} \times _{S}(\overline{Z} \cap H_{\infty})) ), $$

\noindent where $p_{1}$ is projection of $\sE^{d-1} \times_{S} H_{\infty}$ onto the first factor. By construction every point in $\Omega(R)$ consists of a linear map $v=(v_1,\ldots,v_{d}) : \A^N_{S} \rightarrow \A^d_{S} $ such that $L_{v'} \cap \overline{Z} = \emptyset $, where $v'= (v_1,\ldots,v_{d-1})$. By Lemma \ref{shafar}, this will be our required finite map, thus  proving $\Omega(R) \neq \emptyset$ will finish the proposition. As $R$ is Henselian local, the induced map from $\Omega(R)$ to $\Omega(k)$ is surjective, hence it suffices to prove $\Omega(k) =  \Omega_{\sigma}(k) \neq \emptyset $. By construction we have, $\Omega_{\sigma}(k) = \sE^{d}_{\sigma} \setminus p_{1}(V_{\sigma} \cap (\sE^{d}_{\sigma} \times _{S}((\overline{Z} )_{\sigma}\cap H_{\infty})) )$ and any point in $\Omega(k)$ gives a linear map $u=(u_1,\ldots,u_{d}) : \A^N_{k} \rightarrow  \A^d_{k} $ such that $L_{u'} \cap (\overline{Z})_\sigma = \emptyset  $, where $u'= (u_1,\ldots,u_{d-1})$. By Lemma \ref{misses} such a map exists.
\end{proof}

\begin{proposition}\label{er} Let $\phi=(\psi,\nu)=(u_1,\ldots,u_d): X\rightarrow \A^{d-1}_S\times \A^1_S$ and $F:=\psi^{-1}(\psi(z))\cap Z$.
	There exists an open set $\Omega_2\subset \sE^d$ such that $\Omega_2(R)\neq \emptyset$ and for any $\phi\in\Omega_2(R)$
	\begin{enumerate}
\item $\phi$ is \etale{} at all points of $F$.
\item $\phi_{|F}:F\ra \phi(F)$ is radicial.
	\end{enumerate}  
\end{proposition}
\begin{proof} See \cite[Lemma 2.12]{schmidt2018}.
\end{proof}
\begin{proof}[Proof of Theorem \ref{fer}.]
	Let $\Omega_1$ and $\Omega_2$ be as in the Propositions \ref{finite} and \ref{er}. Then the set $\Omega=(\Omega_1\times \sE)\cap\Omega_2$ satisfies all the required conditions.
\end{proof}
Now we obtain the sets $U$ and $V$. The sets $U$ and $V$ are constructed to satisfy all the conditions of Theorem \ref{main}.
\begin{lemma}\label{setv}
	Let $\Phi=(\Psi,\nu)$ satisfy conditions (1)-(3) of Theorem \ref{fer}. Then there exists an open neighborhood $V\subset \A^{d-1}_S$ of $\Psi(z)$ such that
	\begin{enumerate}
		\item $\Phi$ is \etale{} at all points of $Z\cap \Psi^{-1}(V)$.
		\item $\Phi|_{Z\cap \Psi^{-1}(V)}:Z\cap \Psi^{-1}(V)\rightarrow \A^1_V$ is a closed immersion.
	\end{enumerate}
\end{lemma}
	\begin{proof}  See \cite[Lemma 2.13]{schmidt2018}.
	


\end{proof}
\begin{lemma}\label{setu}
There exists a closed subset $\mathfrak{U}\subset \Psi^{-1}(V)$ such that
	\begin{enumerate}
		
		\item $U_1=\Psi^{-1}(V)\setminus\mathfrak{U}$ contains $z$ 
		\item $U_1$ satisfies $Z\cap \Psi^{-1}(V)=Z\cap U_1 $ 
		and $\Phi^{-1}(\Phi(Z\cap U_1)) \cap U_1=Z\cap U_1$.
	\end{enumerate}
\end{lemma}
\begin{proof} See \cite[Lemma 2.14]{schmidt2018}
\end{proof}
\begin{proof}[Proof of Theorem \ref{main}]
	Let $U_2$ be the open locus where $\Phi$ is \etale{}. From Lemma \ref{setv} $z\in U_2$ and $Z\cap \Psi^{-1}(V)\subset U_2$. Now let $U=U_1\cap U_2$, with $U_1$ as in Lemma \ref{setu}. Then $U$ also satisfies conditions $(2)$ and $(3)$ of Lemma \ref{setu}. Furthermore $\Psi_{U}$ is \etale{}. Hence we get $\Phi,\Psi,U,V$ satisfying all the conditions of Theorem \ref{main}.
\end{proof}

\section{Stable Connectivity}   In this section we give a sketch of the proof of Theorem \ref{connectivity}. We do not claim any originality here and all the proofs of the statements in this section can be  found in  \cite[\S 4]{schmidt2018}. Throughout this section $Sm_{S}$ will denote the category of smooth schemes over a given scheme $S$.

\begin{lemma} \label{zero}
	Let $Spec(R)=S$ be a noetherian scheme of finite Krull dimension with a codimension $d$ point $s \in S$. Let $E \in \sS\sH^{s}_{S^1 \geq d+1}(S)$. Then for any $X \in Sm_{S}$ with $X_{s} \neq \emptyset $  and any $f  \in [ \Sigma^{\infty}_{S^{1}} X_{+} ,  L^{\A^1} E ]$ in $ \sS\sH^{s}_{S^1}(S)$, there exists an open subscheme $U$ in $X$ such that $U$ intersects each irreducible component of $X_{s} $ non trivially and $f|_{\Sigma^{\infty}_{S^{1}} U_{+} } = 0$.
\end{lemma}

\begin{proof}
	Let $Z_{i}$'s be the irreducible components of $X_{s}$ . From the proof of \cite[Lemma 4.9 ]{schmidt2018} we obtain  open subschemes $U_{i}$'s of $X$  such that $U_{i} \cap Z_{i} \neq \emptyset$ and $f|_{\Sigma^{\infty}_{S^{1}} (U_{i })_+}  = 0$. Define $U$ to be union of all such $U_{i}$'s. Since the left Quillen functor $\Sigma^{\infty}_{S^{1}}$ gives an adjunction at the level of homotopy category and  $L^{\A^1} E$ (apart from being a fibrant object in $ \sS\sH^{s}_{S^1}(S)$ ) is a spectrum of Nisnevich sheaves, we have $f|_{\Sigma^{\infty}_{S^{1}} U_{+} } = 0$.
\end{proof}

\begin{lemma}\label{dist}
Let $X$ be a smooth irreducible scheme over $S$  and $U$ be a non-empty open subscheme of $X$. Denote by $Z$ the reduced closed subscheme $ X \setminus U$. Suppose Nisnevich locally on $X$ we have the following Nisnevich distinguished square  

\begin{center}
	\begin{tikzcd}
		& U \arrow[d] \arrow[r]&X\arrow[d,"p"]\\
		& \A^{1}_V \setminus p(Z) \arrow[r]&{\A^{1}_V}
	\end{tikzcd}
\end{center}
 where the map $p : X \rightarrow \A^{1}_V$ in $Sm_{S}$  is \etale{}, with $Z \rightarrow V$ finite. Then $\pi^{\A^{1}}_{0}(X/U)=0$.
\end{lemma}
\begin{proof}
	See \cite[Lemma 4.6 and Cor. 4.7]{schmidt2018}.
\end{proof}
Lemmas \ref{zero} and \ref{dist} together give a bound on the drop in connectivity, which is sufficient to prove the connectivity result Theorem \ref{connectivity}, for details see \cite[Prop. 4.5]{schmidt2018}.
\begin{remark}
	Note that if $X, Z $ and $S$ in previous lemma satisfy the conditions stated in Theorem \ref{main},  we obtain the distinguished square of previous lemma and hence  $\pi^{\A^{1}}_{0}(X/U)=0$.  
\end{remark}

\begin{remark} \label{Hensel} To prove stable connectivity we can assume $S$ to be Henselian local by  \cite[ Lemma 4.10]{schmidt2018}
\end{remark}

\begin{proof}[Proof of Theorem \ref{connectivity}]
	We proceed by induction on the dimension of $S$. The case, $dim(S)=0$ follows from \cite{morel2005}. By Remark \ref{Hensel}, we may assume $S$ to be Henselian local with closed point $\sigma$. Further we can assume $X_{\sigma} \neq \emptyset$, where $X \in Sm_{s}$. Consider $f  \in [ \Sigma^{\infty}_{S^{1}} X_{+} ,  L^{\A^1} E ]$, then by  Lemma \ref{zero} we obtain an open subscheme $U$ such that $U$ intersects each irreducible component of $X_{\sigma} $ non-trivially and $f|_{\Sigma^{\infty}_{S^{1}} U_{+} } = 0$. Take the reduced closed subscheme $Z = X \setminus U$. Then dim $Z_{\sigma} < $ dim $X_{\sigma} $. Hence by Gabber presentation lemma we have Nisnevich distinguished square of Lemma \ref{dist} which proves $\pi^{\A^{1}}_{0}(X/U)=0$. Now connectivity follows from \cite[Prop. 4.5]{schmidt2018}
\end{proof}
\bibliographystyle{alphanum}
\bibliography{biblio}

	\vspace{0.8cm}

\begin{center}
	Neeraj Deshmukh, IISER Pune, Dr. Homi Bhabha Road, Pashan, Pune : 411008, INDIA\\
	email: neeraj.deshmukh@students.iiserpune.ac.in\\
	\vspace{0.3cm}

	Amit Hogadi, IISER Pune, Dr. Homi Bhabha Road, Pashan, Pune : 411008, INDIA\\
	email: amit@iiserpune.ac.in\\
	\vspace{0.3cm}
	Girish Kulkarni, IISER Pune, Dr. Homi Bhabha Road, Pashan, Pune : 411008, INDIA\\
	email: girish.kulkarni@students.iiserpune.ac.in\\
		\vspace{0.3cm}
		Suraj Yadav, IISER Pune, Dr. Homi Bhabha Road, Pashan, Pune : 411008, INDIA\\
	email: surajprakash.yadav@students.iiserpune.ac.in\\

\end{center}

\end{document}